\documentclass[A4paper]{article} 
\usepackage{amsmath}
\usepackage{amscd}
\usepackage{amsthm}
\usepackage[mathscr]{eucal}
\usepackage{amssymb} 
\usepackage{latexsym}
\usepackage{amsthm}
\theoremstyle{definition}
\newtheorem{theorem}{Theorem}[section]

\newtheorem{lemma}[theorem]{Lemma}
\newtheorem{proposition}[theorem]{Proposition}
\newtheorem{corollary}[theorem]{Corollary}
\newtheorem{remark}[theorem]{Remark}
\newtheorem{example}[theorem]{Example}

\newtheorem{conjecture}{Conjecture}
 
\title{Eigenvalues and Eigenfunctions of Double Layer Potentials}
\author{{\Large Yoshihisa Miyanishi and Takashi Suzuki} \vspace{2mm}\\
Department of Systems Innovation, \\ Graduate School of Engineering Science,\\ Osaka University \\
Toyonaka 560-8531, Japan}
\begin{document}
\maketitle
\begin{abstract} 
Eigenvalues and eigenfunctions of two- and three-dimensional double layer potentials are considered. 
Let  $\Omega$ be a $C^2$ bounded region in $\mathbf{R}^n$ $(n=2, 3)$. The double layer 
potential $K: L^2(\partial \Omega) \rightarrow L^2(\partial \Omega) $ is defined by     
$$
(K \psi)(x) \equiv  \int_{\partial \Omega} \psi(y)\cdot \nu_{y} E(x, y) \; ds_y,  
$$
where
$$
E(x, y)=
\begin{cases} 
\frac{1}{\pi} \log\frac{1}{|x-y|} \quad \mbox{if}\; n=2, \\
\frac{1}{2\pi} \frac{1}{|x-y|} \quad\hspace{4mm} \mbox{if}\; n=3,  
\end{cases}
$$ 
$ds_y$ is the line or surface element and $\nu_y$ means the outer normal derivative on $\partial \Omega$. 
It is known that $K$ is a compact operator on $L^2(\partial \Omega)$ and consists of at most a countable 
number of eigenvalues, with $0$ the only possible limit point.
The aim of this paper is to establish some relationships between eigenvalues, 
eigenfunctions and the geometry of $\partial\Omega$.  
\end{abstract}
\renewcommand{\thefootnote}{}
\section{Introduction and Results}
\par\hspace{5mm} 
Let $\Omega$ be a $C^2$ bounded region in $\mathbf{R}^n\; (n=2, 3)$. 
Consider the double layer 
potential $K: L^2(\partial \Omega) \rightarrow L^2(\partial \Omega) $:    
$$
(K \psi)(x) \equiv \int_{\partial \Omega} \psi(y)\cdot 
\nu_{y} E(x, y) \; ds_y,  
$$
where
$$
E(x, y)=
\begin{cases} 
\frac{1}{\pi} \log\frac{1}{|x-y|} \quad \mbox{if}\; n=2, \\
\frac{1}{2\pi} \frac{1}{|x-y|} \quad\hspace{4.5mm} \mbox{if}\; n=3,  
\end{cases}
$$ 
$ds_y$ is the line or surface element 
and $\nu_y$ means the outer normal derivative on $\partial \Omega$. 
We know that $K$ is a compact operator 
on $L^2(\partial \Omega)$ and consists of at most a countable 
number of eigenvalues, with $0$ the only possible limit point. 
It is also known that the eigenvalues of the double layer 
potential integral operator lie in the interval $[-1, 1)$ 
and the eigenvalue $-1$ corresponds to constant 
eigenfunctions {(See \cite{Pl} and see also \cite{Ta} 
for some recent progress)}. 
\par
We set the ordered eigenvalues and eigenfunctions 
counting multiplicities by      
$$
\sigma_p(K)
=\{\; \lambda_j \mid |\lambda_0|> |\lambda_1| \geq |\lambda_2|\geq\; \cdots\}, 
$$ 
where 
$$
K e_{\lambda_j}(x) =\lambda_j e_{\lambda_j}(x). 
$$
\par 
Recall also that every compact operator $K$ on Hilbert space takes the 
following canonical form 
$$
K\psi  = \sum_{j=1}^{\infty} \alpha_{j} \langle \psi, v_j \rangle  u_j 
$$ 
for some orthonormal basis $\{u_j\}$ and $\{ v_j\}$, 
where $\alpha_i$ are sigular values of $K$ (i.e. the eigenvalues of $(K^*K)^{1/2}$)  
and $\langle \cdot, \cdot \rangle$ means the $L^2(\partial \Omega)$ inner product. 
The singular values are non-negative and we denote the ordered singular values by 
$$
\sigma_{sing}(K)=\{\; \alpha_j \mid \alpha_1\geq \alpha_2 \geq \alpha_3 \geq\; \cdots \}. 
$$
With this in mind, our main concerns are two natural questions: 
\vspace{2mm}\par
(i)\;\; What can we say about the geometry of $\partial \Omega$ given the eigen or singular values? 
\par
(ii)\; What can we say about the eigenvalues, singular values and eigenfunctions given the geometry?   
\vspace{2mm}\par
What we will attempt to prove in this paper are some selected aspects of those two questions: 
isoperimetric eigen and singular value problems, decay rates of eigen and singular values, 
and nodal sets of eigenfunctions. 
Note that these questions are taken from the questions of the 
spectral geometry for elliptic operators. As will be mentioned in the last of this section, 
there are many studies in this direction.  
In other words our aim is to develop the spectral geometry of double layer potentials. 
\par
For this purpose, in \S 2 we start studying two questions for two dimensional double layer potentials: 
\par (Q1)\ What types of eigen and singular values give the isoperimetric propertiy 
of $\partial \Omega$? 
\par
(Q2)\ What types of sequences can occur as eigen and singular values? 
\par An answer of (Q1) is given by:
\par
\setcounter{section}{2}
\setcounter{theorem}{6}
\begin{theorem}
Let $n=2$ and $\Omega$ be a simply connected region with $C^2$ boundary. Then 
$$\sigma_{sing}(K)\backslash\{0\}=\{1 \}\ \mbox{is necessary and sufficient for}\ \partial \Omega=S^1.$$ 
\end{theorem} 
It follows that 
$(K\psi)(x)=0$ for all $\psi(x)\in L^2_0(\partial \Omega)$, then $\partial \Omega=S^1$ (See Corollary 2.8).
There are some proofs of this theorem (See e.g. \cite{Li}). 
In \S 2.1, a short alternative proof is given 
by using Hilbert-Schmidt norm of $K$. For $\partial \Omega=S^1$, 
$K$ has an eigenvalue $-1$ of multiplicty $1$ and an eigenvalue $0$ of 
infinite countable multiplities (See \cite{Ah1} and \S 2.1).   
So the condition $\sigma_{sing}(K)\backslash \{0\}=\{1\}$ can be 
replaced by $\sigma_{sing}(K)=\{1, 0\}$. 
Moreover the theory of quasi-conformal mapping states that 
$\sigma_{p}(K)=\{-1, 0\}$ is also necessary and sufficient for 
$\partial \Omega=S^1$. We mention about this result 
in the last of \S2.1. 
\par 
To answer (Q2), we consider Schatten norm of $K$ and estimate a decay rate
\renewcommand{\thefootnote}{[1]}
\footnote[1]{When considering the asymptotics of eigenvalues and eigenfunctions, 
we henceafter assume that the eigenvalues counting multiplisities are infinite.}  
of eigen and singular values by the regularity of $\partial \Omega$:  
\setcounter{section}{2}
\setcounter{theorem}{11}
\begin{theorem}
Let $n=2$ and $\Omega$ be a $C^{k}$ ($k\geqq 2$) region. For any $\alpha >-2k+3$,   
$$
\alpha_j=o(j^{\alpha/2})\ \mbox{and}\ \lambda_j = o(j^{\alpha/2}) \quad \text{as}\; j\rightarrow \infty  
$$
where $o$ means the small order. 
\end{theorem}
It follows that if $\Omega$ be a $C^{\infty}$ region, we have $\lambda_j = o(j^{-\infty})$. 
For an ellipse $\partial\Omega$, for instance,  
direct calculations give $\lambda_j = O(e^{-cj})$ (See \cite{Ah1}, \cite[\S 8.3]{KPS} 
and example 2.2).  
It should be emphasized that the ellipses are analytic curves and so 
the eigenvalues are presumed to have the stronger decay properties 
than the case of smooth curves. These viewpoints shed some new lights on 
eigenvalue asymptotics.
\par Few studies have focused on the eigenfunctions. 
In \S 3, we show a question for two dimensional double layer potentials: 
\par 
(Q3)\; What can we say about the nodal sets of eigenfunctions? 
\par 
Here we establish the holomorphic extention of $e_{\lambda_j}$ 
for analytic curves and give the growth of zeroes of analytic eigenfunctions:  
\setcounter{section}{3}
\setcounter{theorem}{9}
\begin{theorem}
Let $n=2$ and $\Omega$ be a real analytic region 
and $\{e_{\lambda_j}(x)\}\subset C^{\omega}(\partial \Omega)$ be real analytic eigenfunctions. 
There exists $C>0$, depending only on $\Omega$,  
such that the zeroes $N(e_{\lambda_j}(x))$ satisfy   
$$\sharp N(e_{\lambda_j}(x)) \leqq C |\log |{\lambda_j}||.$$
\end{theorem}
From Theorem 3.10 one can expect that the positive eigenfunctions correspond to   
eigenvalue $-1$, so the positive eigenfunction is constant. 
This fact holds true even for a much more general case (See Theorem 3.1). 
\vspace{3mm}\par 
\setcounter{section}{1}
\setcounter{theorem}{6}
\par 
Apart from $n=2$, for the case of $n=3$, the analogy of the above theorems 
is difficult to handle. So we shall discuss only some remarks and conjectures 
in \S 4. The behavior of $\sigma_p(K)$, for instance, changes to be drastic:   
\setcounter{section}{4}
\setcounter{theorem}{0}
\begin{remark}
Let $n=3$ and $\Omega$ be a smooth region. For $\alpha>-\frac{1}{2}$, 
we have 
$$
\lambda_{j}=o({j^{\alpha}}) \quad \text{as}\; j\rightarrow \infty. 
$$
\end{remark}
\par 
For $n=3$, no satisfactory answer of isoperimetric eigenvalue problems 
has been found yet. Instead, in the context of studying the eigenvalue problems 
we propose reasonable conjectures:  
\setcounter{section}{3}
\setcounter{theorem}{2}
\begin{conjecture} Let $n=3$ and $\underline{\lambda} \equiv 
{\rm min}\; \sigma_p(K)\backslash \{-1\}$.  
We have  
$$
\sup\limits_{\partial \Omega} \underline{\lambda}=-\frac{1}{3}  
$$
where the supremum is taken over all $C^{\infty}$ simply 
connected closed surfaces. The supremum is achieved if and only if $\partial \Omega=S^2$. 
\end{conjecture}
\begin{conjecture}
Let $n=3$. For $p>1$, we have  
$$
\inf\limits_{\partial \Omega} {\rm tr}\{(K^* K)^p\}=\left(1-\frac{1}{2^{2p-1}}\right)\zeta (2p-1)
$$
where the infimum is taken over all $C^{\infty}$ simply 
connected closed surfaces and $\zeta(s)$ denotes the Riemann zeta function. 
The infimum is achieved if and only if $\partial \Omega=S^2$. 
\end{conjecture}
We confirm the validity of these conjectures. When $C^{\infty}$ closed surfaces are replaced by 
ellipsoids, these conjectures will be proved (See Theorem 4.3). 
\par 
We end the introduction by comparing with the above results and the spectral geometry 
of Laplacian on manifolds. In the case of Laplacian, 
the isoperimetric properties of manifolds are characterized by 
the first eigenvalue or second eigenvalue or eigenvalue asymptotics, etc.
(See e.g. \cite{Be} and references therein). 
Theorem 2.7, Conjecture 1 and Conjecture 2 correspond to 
these results. Theorem 2.12 can be viewed as eigenvalue asymptotics 
called Weyl's law. For Laplacian, Weyl's law includes the information about the dimension and 
volume of manifolds, and etc.(See e.g. \cite{CH} and \cite{ANPS} and references therein). 
Theorems about zeroes of Laplace eigenfunctions are known as Courant's nodal line theorem 
and Donnelly-Fefferman's results and etc.(See e.g. \cite{CH}, \cite{DF} and \cite{Ze}). 
Roughly speaking, they estimate 
the Haussdorf dimension and measure of nodal sets by the eigenvalues. 
Indeed we prove Theorem 3.10 by using the modified 
Donnelly-Fefferman value distribution theory.  
\setcounter{section}{1}
\setcounter{theorem}{0}
\section{Eigenvalues and singular values of two dimensional double layer potentials}
In \S2 we shall restrict ourselves to two dimensional double layer potentials. 
Such a situation allows us to treat Hilbert-Schmidt norm and Schatten norm of $K$. 
Using these norms, we obtain isoperimetric properties of singular values in \S2.1 
and decay estimates in \S2.2.  
\par
\subsection{The trace of $K^*K$ and its application to isoperimetric problems}
\par
\hspace{5mm} 
We consider the boundary integral equation: 
$$
(K\psi) (x) \equiv \frac{1}{\pi} 
\int_{\partial \Omega} \psi(y)\cdot \nu_{y} \log\frac{1}{|x-y|}\; ds_y,  \eqno{(1)}   
$$
where $\Omega$ is a $C^2$ bounded region in ${\mathbf{R}}^2$ 
and $\nu_y$ means the outer normal derivative on $\partial \Omega$. 
$\partial \Omega \in C^2  \subset C^{1,\alpha}$ is a Lyapunov curve and 
$K$ as well as $K^*$ are compact operators on $L^2(\partial \Omega)$.  
Moreover the spectra in $L^2(\partial \Omega)$ and 
in $C^0(\partial \Omega)$ are identical (e.g. \cite[Theorem 7.3.2]{Mik}). 
A standard result in two-dimensional 
potential theory (See \cite[p.78-80]{Tr} and see also Lemma 2.12) states that 
for closed $C^2$ curves $\partial \Omega$  
$$
\lim_{\substack{x\rightarrow y \\[1pt] x \in \partial \Omega}} \nu_{y} \log\frac{1}{|x-y|}
=-\frac{1}{2}\kappa(y), 
$$
where $\kappa(y)$ denotes the curvature of $\partial \Omega$. 
Consequently, unlike the singular nature of the double layer potentials in $\mathbf{R}^3$, 
the double layer kernel in $\mathbf{R}^2$ is continuous for all 
points $x$ and $y$ on $\partial \Omega$, including when $x=y$. 
It is also known the eigenvalues of the integral operator $K$, defined in equation (1), 
lie in the interval $[-1, 1)$ and are symmetric with respect to the origin 
(e.g. \cite{BM}, \cite{Sh}). The only exception is the eigenvalue $-1$ corresponding 
to constant eigenfunctions. Summarizing these results, we have 
the ``formal" trivial trace formula 
\renewcommand{\thefootnote}{[2]}
\footnote[2]{$K$ is not always 
the usual trace class operator. The above ``formal" trace formula is defined by   
only a conditional summation. Note that 
if $\partial \Omega$ is ${C^3}$ curve, $K$ is the usual trace class operator 
(See the proof of Theorem 2.12 and Remark 2.16).} 
for $K$:  
$$
{\rm tr}(K)\equiv\hspace{-8mm} 
\sum_{\begin{subarray}{c}{\lambda_i : {\rm eigenvalue}\ {\rm of}\ K} \\[1pt] 
{|\lambda_0|>|\lambda_1|=|\lambda_2| \geq |\lambda_3|=|\lambda_4| \geq \cdots }\end{subarray}}
\hspace{-10mm} \lambda_i
=\int_{\partial \Omega} -\frac{1}{2 \pi}\kappa(y) d{s_y}=-1. 
$$
Here $K$ is not selfadjoint or even normal, 
but $K^*K$ is a selfadjoint trace class operators. 
Thus the trace of $K^*K$: 
$$
{\rm{tr}} (K^*K)   
$$ 
is also considered. Consequently we obtain some asymptotic properties of the singular values of $K$. 
In \S 2, we start out by rapidly going over basic examples of ${\rm{tr}} (K^*K)$.    
\par
\begin{example}[The circle ({See \cite{Ah1}})] Let $\partial \Omega$ be a circle of radius $R$. 
We find 
\begin{align*}
&\sigma_p (K)=\{-1, 0\},\\ 
&{\rm{tr}}(K^*K)=1,  
\end{align*}
where $\sigma_p (K)$ means the set of eigenvalues of $K$. 
\end{example}
In the case of ellipse, we have $ {\rm{tr}}(K^*K)>1$.  
\begin{example}[The ellipse ({See \cite{Ah1}, \cite[\S 8.3]{KPS}})]
For $R>0$ and $c>0$, we define the ellipse by 
$\partial \Omega=\{ (x,y) |\ x=\frac{1}{2}c \cosh R \cos \theta,
\; y=\frac{1}{2}c \sinh R \sin \theta \}$. Then   
\begin{align*}
&\sigma_p (K)=\{-1, \pm e^{-2mR}\ |\ m\in {\mathbf{N}} \},\\ 
&{\rm{tr}}(K^*K)>1. 
\end{align*}
\end{example}
Seeing this, we want to characterize the region of which ${\rm{tr}} (K^*K)=1$. 
\begin{lemma}
$K^* K$ is a trace class operator on $L^2(\partial \Omega)$,  
i.e., $K$ is a Hilbert-Schmidt class operator  
and 
$${\rm{tr}}(K^*K)=\frac{1}{\pi^2}\int_{\partial \Omega_y}
\int_{\partial \Omega_x} |\nu_y \log|x-y||^2 \; ds_x\; ds_y. $$
\end{lemma}
\begin{proof}
$$(KK^* \psi)(x)=\frac{1}{\pi^2} \int_{\partial \Omega_z}
\int_{\partial \Omega_y}  \left(\nu_z \log|x-z| \right) 
\left(\nu_z \log|y-z| \right) \psi(y)  \; ds_y\; ds_z. $$
The kernel $K(x, y)$ of $KK^*$ is continuous symmetric 
non-negative definite on $\partial \Omega_x \times \partial \Omega_y$. 
By Mercer's theorem (See e.g. \cite[p.138]{CH}, \cite[Theorem 1.1]{FM} and \cite{Ko}), 
there is an orthonormal set $\{u_i\}_i$ of $L^2(\partial \Omega)$ 
consisting of eigenfunctions of $KK^*$ such that corresponding eigenvalues $\{\mu_i\}_i$ 
are nonnegative. The eigenfunctions corresponding to non-zero eigenvalues are 
continuous on $\partial \Omega$ and $K(x, y)$ has the representation 
$$
K(x, y)=\sum_{i=1}^{\infty} \mu_i u_i(x) u_i(y), 
$$ 
where the convergence is absolute and uniform.  
This leads to 
$${\rm{tr}}(K^*K)={\rm{tr}}(KK^*)=\frac{1}{\pi^2} \int_{\partial \Omega_y}
\int_{\partial \Omega_x} |\nu_y \log|x-y||^2 \; ds_x\; ds_y. $$
\end{proof}
Recall that every compact operator $K$ on Hilbert space takes the 
following canonical form 
$$
K\psi  = \sum_{j=1}^{\infty} \alpha_{j} \langle \psi, v_j \rangle  u_j 
$$ 
for some orthonormal basis $\{u_j\}$ and $\{ v_j\}$, 
where $\alpha_i$ are sigular values of $K$ (i.e. the eigenvalues of $(K^*K)^{1/2}$)  
and $\langle \cdot, \cdot \rangle$ means the $L^2$ inner product. 
Also, the usual operator norm is $\Vert K \Vert=\sup_{j} (\alpha_j)$. 
${\rm{tr}}(K^*K)=\sum_{j=1}^{\infty} |\alpha_j|^2$ and  
by using Weyl's inequality (See. e.g. \cite{Si}, \cite{Te}) :   
$$\sum_{j=1}^{\infty} |\alpha_{j}|^2 \geqq \sum_{\lambda_j\in \sigma_p(K)} |\lambda_j|^2, $$
we obtain: 
\begin{lemma}
$$
{\rm{tr}}(K^*K) \geqq \Vert K \Vert^2, \quad  {\rm{tr}}(K^*K) 
\geqq \sum_{\lambda_j \in \sigma_p(K)} |\lambda_j|^2\geqq 1. 
$$
\end{lemma} 
\begin{remark}
From Lemma 2.4 the ordered eigenvalues satisfy  
$$
\sum_{\lambda_j\in \sigma_p(K)} |\lambda_j|^2<\infty.
$$
So for all $\epsilon>0$ there exists $N$ such that 
$$
(n-N)|\lambda_n|^2\leqq \sum_{N+1}^{n} |\lambda_j|^2<\epsilon
$$
and hence 
$$
n|\lambda_n|^2<2\epsilon \quad \mbox{for all}\; n>2N.  
$$
Accordingly $\lambda_j=o(j^{-1/2})$. 
This is not the best possible estimate (See example 2.1, 2.2 and see also \S 2.2).  
\end{remark}
In the following of this subsection, we apply the trace for 
the analysis of singular values.   
The minimizer of ${\rm tr}{(K^*K)}$ is attained by $\partial \Omega=S^1$. 
\begin{theorem} Let $\Omega$ be a simply connected region with $C^2$ 
boundary. 
$$
 {\rm{tr}}(K^*K)=1\ \mbox{is necessary and sufficient for}\ \partial \Omega=S^1.
$$
\end{theorem}
\begin{proof}
We note that $\int_{\partial \Omega_x}\nu_y \log|x-y| \; ds_y=\pi$. 
Letting $C\equiv \pi\cdot ({\rm length\ of\ }\partial \Omega)^{-1}$, 
\begin{align*} 
{\rm{tr}}(K^*K) &=\frac{1}{\pi^2} \int_{\partial \Omega_y}
\int_{\partial \Omega_x} |\nu_y \log|x-y||^2 \; ds_x\; ds_y \\
&=\frac{1}{\pi^2} \int_{\partial \Omega_y}\int_{\partial \Omega_x} 
|\nu_y \log|x-y| -C|^2 \; ds_x\; ds_y 
+ \frac{2C}{\pi^2}\int_{\partial \Omega_y} \int_{\partial \Omega_x} 
\nu_y \log|x-y| \; ds_x\; ds_y 
- \frac{C^2}{\pi^2} \int_{\partial \Omega_y} \int_{\partial \Omega_x}  \; ds_x\; ds_y \\
&=\frac{1}{\pi^2} \int_{\partial \Omega_y}\int_{\partial \Omega_x} 
|\nu_y \log|x-y| -C|^2 \; ds_x\; ds_y +1. 
\end{align*}
It follows that 
${\rm{tr}}(K^*K)=1 \; \Rightarrow \; \nu_y \log |x-y| =C$  
for all $(x, y)\in \partial\Omega_x \times \partial \Omega_y$. 
\par
By the continuity of $\nu_y \log|x-y|$, 
$$\frac{1}{2}\kappa(x)=C\quad {\rm (constant)}$$ 
as desired. 
\end{proof}
Suppose ${\rm{tr}}(K^*K)=1$. Then only one singular value takes $1$, otherwise $\alpha_j=0$.
Thus we obtain: 
\begin{theorem}
Let $\Omega$ be a simply connected region with $C^2$ 
boundary. 
$$ \sigma_{sing}(K)\backslash\{0\}=\{1 \}\ \mbox{is necessary and sufficient for}\ \partial \Omega=S^1$$ 
where $\sigma_{sing}(K)$ denotes the set of singular values. 
\end{theorem} 
From the canonical form of $K$, the ball symmetry property of double layer potentials is also obtained: 
\begin{corollary}(See \cite[Theorem 1.3]{Li}) Let $\Omega$ be a simply connected region with $C^2$ 
boundary 
and $L^2_0(\partial \Omega)=\{ \psi\in L^2(\partial \Omega)\; |\; \int_{\partial\Omega}\psi\; ds=0\}.$ 
If $(K\psi)(x)=0$ for all $\psi(x)\in L^2_0(\partial \Omega)$, then $\partial \Omega=S^1$. 
\end{corollary} 
M. Lim proved that if $K$ is self-adjoint, then $\partial\Omega$ is circle. 
Lim's work is esentially based on the ``moving hyperplane" method of Alxandroff and Serrin, 
but we are not aware of any studies in this derection (See \cite{Re}, \cite{Se}). 
It is also known \cite{S} that the disk is the only planar domain for which $K$ has finite rank. 
\par 
\begin{remark}
For higher dimensions, $K^*K$ is not always a trace class operator (See \S 4). 
\end{remark}
In the following we introduce 
well-known classical results on $\sigma_p(K)$ (See e.g. \cite{Scho1}):  
Even when $\sigma_{sing}(K)$ is replaced by $\sigma_{p}(K)$, Theorem 2.7 holds true.
The points $\lambda \in \sigma_p(K) \backslash \{-1\}$ are known as the Fredholm 
eigenvalues of $\partial\Omega$. 
The largest eigenvalue $\overline{\lambda}$ is often interest. 
By the symmetry of eigenvalues, we have  
$\overline{\lambda}=-\underline{\lambda}\equiv -\inf{\sigma_{p}(K)\backslash\{-1\}}$. 
\par
Let $\partial\Omega$ be on the Riemann sphere $\hat{\mathbf{C}}=\mathbf{C}\cup \{\infty\}$.
Then $\partial\Omega$ divides $\hat{\mathbf{C}}=\mathbf{C}\cup \{\infty\}$ into complementary 
simply connected domains $\Omega$ and $\bar{\Omega}^c$. Let $\mathcal{H}$ be the family of all 
functions $u$ continuous in $\hat{\mathbf{C}}$ and harmonic in $\Omega\cup \bar{\Omega}^c$, 
with $0<\mathcal{D}_{\Omega}(u)+\mathcal{D}_{{\bar{\Omega}}^c}(u)<\infty$. 
Here $\mathcal{D}_{A}(u)$ denotes the Dirichlet integral on $A$: 
$$\mathcal{D}_{A}(u)=\int \int_{A} u_x^2+u_y^2\; dx dy.$$
Ahlfors \cite{Ahl} showed the relationships between the Fredholm eigenvalues, the Dirichlet integral 
and quasiconformal mappings. 
Especially the value $\overline{\lambda}$ 
can be represented in terms of the Dirichlet integral:  
$$\overline{\lambda}
=\sup_{u\in {\mathcal{H}}} 
\frac{|\mathcal{D}_{\Omega}(u)-\mathcal{D}_{{\bar{\Omega}}^c}(u)|}
{\mathcal{D}_{\Omega}(u)+\mathcal{D}_{{\bar{\Omega}}^c}(u)}.$$ 
Since conformal mappings preserve harmonic functions and Dirichlet integrals, 
$\overline{\lambda}$ is invariant under linear fractional transformations. 
Let $f : \Omega\cup\bar{\Omega}^c  \rightarrow \Omega\cup\bar{\Omega}^c$ of $\hat{\mathbf{C}}$ be 
a orientation preserving homeomorphism whose distributional partial derivatives 
are in $L^2_{loc}$. 
If $f$ preserves the curve $\partial \Omega$, the reflection coefficient of $f$ is defined by 
$$q_{\partial \Omega}=\inf \Vert \partial_{\bar{z}} f/\partial_z f \Vert_{\infty}$$ 
where the infimum is taken over all quasireflections across $\partial\Omega$ provided 
these exist and is attained by some quasireflection $f_0$. 
The number $M$ satisfying  
$$
q_{\partial \Omega}=\frac{M+1}{M-1}  
$$ 
is called the quasiconformal constant. The $M$-quasiconformal mapping  
is an orientation-preserving diffeomorphism whose derivative maps 
infinitesimal circles to infinitesimal ellipses with eccentricity at most $M$. 
A basic ingredient for estimating $\overline{\lambda}$ is known 
as Ahlfors inequality \cite{Ahl}:  
$$
\overline{\lambda}\geq \frac{1}{q_{\partial \Omega}}.  
$$ 
If $\overline{\lambda}=0$, then $q_{\partial\Omega}=\infty$ and $M=1$. 
So $f_0$ is 1-conformal, hence conformal. The conformal mapping of $\Omega\cup\bar{\Omega}^c$ 
onto $\Omega\cup\bar{\Omega}^c$ can be extended to a 1-conformal mapping of $\hat{\mathbf{C}}$ 
onto $\hat{\mathbf{C}}$. The only such mappings are linear fractional trandformations, 
and so, since $\partial\Omega$ is mapped onto $\partial \Omega$, it must itself be $S^1$. 
Thus $q_{\partial \Omega}=\infty$, $\overline{\lambda}=0$ and $\sigma_p(K) \backslash \{-1\}=\{0\}$
only for the circle. 
\begin{remark}
Many authors study in this direction. We mention only some fascinating results. 
\par
1. If $\partial\Omega$ is convex, then $\overline{\lambda}\geqq \{1-(|\partial\Omega|/2\pi R)\}^{-1}$ 
where $R$ is the supremum of radii of all circles which intersect 
$\partial\Omega$ at least 3 points. (In case $\partial\Omega$ is 
smooth, $R$ is the maximum radius of curvature). 
This is due to C. Neumann (e.g. \cite{Scho2}, \cite{Wa}).   
\par
2. Recently Krushkal proved the 
celebrated inequality (See \cite[p.358]{Kr1} and reference in \cite{Kr2}): 
$$
\frac{3}{2\sqrt{2}}\frac{1}{q_{\partial \Omega}} \geq 
\overline{\lambda} \geq \frac{1}{q_{\partial \Omega}}.  
$$ 
\par
3. For higher dimensions, Fredholm eigenvalues are also characterized by 
Dirichlet integrals (e.g. \cite{S}, \cite{KPS}).    
\end{remark}      
\begin{remark}
Taking the limit $R\rightarrow \infty$ in example 2.2, we have 
$\sup\limits_{\partial \Omega} \overline{\lambda}=1$ 
where the supremum is taken over all $C^{\infty}$ domain $\Omega$. 
\end{remark}
\subsection{Asymptotic properties of $\sigma_{p}(K)$}
In the preceding subsection, we considered Hilbert-Schmidt norm of $K$. More generally 
$K$ is in Schatten classes of $r>\frac{2}{2k-3}$ 
for $C^k$ $(k \geqq 2)$ closed curve $\partial \Omega$. 
(For details on the notion of the Schatten classes, see e.g. \cite{Mc}).
\par 
Let $\lambda_n$ be eigenvalues of $K$ satisfying 
$$
|\lambda_0|>|\lambda_1|=|\lambda_2| \geq |\lambda_3|=|\lambda_4| \geq \cdots. 
$$
In the case of ellipse, we find 
$$
\lambda_j =O(e^{-cj})
$$
where $\lambda_j=O(e^{-cj})$ means that there exists a constant $C>0$ such that 
$\lambda_j\leq C e^{-cj}$ for large $j\in \mathbf{N}$. 
For general $C^k$ closed curves $\partial \Omega$, we obtain: 
\begin{theorem}
Let $n=2$ and $\Omega$ be a $C^{k}$ $(k \geqq 2)$ bounded region. For any $\alpha >-2k+3$,   
$$\alpha_j=o(j^{\alpha/2})\ \mbox{and}\ \lambda_j =o(j^{\alpha/2})\quad \text{as}\; j\rightarrow \infty.$$ 
\end{theorem}
Thus the boundary regularity is essential to the decay rate of eigenvalues. 
To prove Theorem 2.12, we first prepare a fundamental lemma. 
For the sake of the readers' convenience, we also give the proof to the following.   
\begin{lemma}
If $k \geqq 2$, then $E\in C^{k-2}(\partial \Omega \times \partial \Omega)$. Especially we have 
$$
\lim_{\substack{x\rightarrow y \\[1pt] x \in \partial \Omega}} \nu_{y} \log\frac{1}{|x-y|}
=-\frac{1}{2}\kappa(y), 
$$
where $\kappa(y)$ denotes the curvature of $\partial \Omega$. 
\end{lemma} 
\begin{proof}
For every point $P$ on $\partial \Omega$ there exists a small neighborhood $B_\epsilon(P)$ 
such that the part of $B_{\epsilon}(P) \cap \partial \Omega$ for some orientation of 
the axes of coordinate system $(\xi, \eta)$, admits a representation (See Fig.1)
\par 
\begin{minipage}{0.65\hsize}
$$
\partial \Omega \cap B_{\epsilon}(P)=\{(\xi, \eta)\ |\ \eta=F(\xi),\ |\xi|<\epsilon' \}
$$ 
where $F(\xi)\in C^{k}$. 
\par
For $x=(\xi_{1}, \eta_{1})$ and $y=(\xi_{2}, \eta_{2})$,  
$\nu_{y}$ and $\log|x-y|$ is given by   
\begin{align*}
&\nu_{y} =\Big(\frac{F'(\xi_2)}{\{1+(F'(\xi_2))^2\}^{1/2}}\frac{\partial}{\partial \xi_2},
                            \ \frac{-1}{\{1+(F'(\xi_2))^2\}^{1/2}}\frac{\partial}{\partial \eta_2}\Big),  \\
&\log |x-y| =\frac{1}{2}\log\{(\xi_1-\xi_2)^2+(\eta_{1}-\eta_{2})^2\}.         
\end{align*}
\end{minipage}
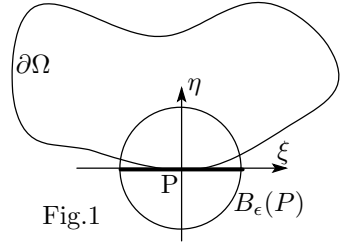
\begin{minipage}{0.4\hsize} 
{\unitlength 0.1in%
\begin{picture}( 17.0900, 12.4900)(  7.9100,-20.7000)%
%
\special{pn 8}%
\special{pa 1193 1592}%
\special{pa 1226 1602}%
\special{pa 1258 1613}%
\special{pa 1291 1623}%
\special{pa 1323 1633}%
\special{pa 1356 1642}%
\special{pa 1388 1651}%
\special{pa 1421 1660}%
\special{pa 1453 1668}%
\special{pa 1485 1675}%
\special{pa 1517 1681}%
\special{pa 1548 1687}%
\special{pa 1580 1692}%
\special{pa 1611 1695}%
\special{pa 1642 1697}%
\special{pa 1672 1699}%
\special{pa 1703 1698}%
\special{pa 1733 1697}%
\special{pa 1762 1694}%
\special{pa 1792 1689}%
\special{pa 1821 1683}%
\special{pa 1849 1675}%
\special{pa 1878 1666}%
\special{pa 1934 1644}%
\special{pa 1962 1631}%
\special{pa 1990 1617}%
\special{pa 2046 1587}%
\special{pa 2075 1570}%
\special{pa 2131 1536}%
\special{pa 2189 1500}%
\special{pa 2218 1481}%
\special{pa 2248 1462}%
\special{pa 2278 1444}%
\special{pa 2309 1425}%
\special{pa 2369 1387}%
\special{pa 2397 1367}%
\special{pa 2423 1347}%
\special{pa 2447 1325}%
\special{pa 2467 1302}%
\special{pa 2483 1278}%
\special{pa 2494 1253}%
\special{pa 2500 1226}%
\special{pa 2501 1197}%
\special{pa 2497 1168}%
\special{pa 2489 1137}%
\special{pa 2477 1107}%
\special{pa 2461 1076}%
\special{pa 2443 1045}%
\special{pa 2421 1015}%
\special{pa 2396 986}%
\special{pa 2369 959}%
\special{pa 2340 933}%
\special{pa 2310 909}%
\special{pa 2278 887}%
\special{pa 2246 867}%
\special{pa 2212 851}%
\special{pa 2178 838}%
\special{pa 2145 828}%
\special{pa 2112 823}%
\special{pa 2079 821}%
\special{pa 2047 823}%
\special{pa 2015 828}%
\special{pa 1984 836}%
\special{pa 1953 847}%
\special{pa 1923 859}%
\special{pa 1892 873}%
\special{pa 1862 887}%
\special{pa 1833 903}%
\special{pa 1803 918}%
\special{pa 1774 934}%
\special{pa 1744 948}%
\special{pa 1715 962}%
\special{pa 1686 974}%
\special{pa 1656 984}%
\special{pa 1627 992}%
\special{pa 1597 997}%
\special{pa 1567 999}%
\special{pa 1537 999}%
\special{pa 1506 996}%
\special{pa 1475 990}%
\special{pa 1444 983}%
\special{pa 1412 974}%
\special{pa 1380 963}%
\special{pa 1347 951}%
\special{pa 1313 938}%
\special{pa 1279 923}%
\special{pa 1245 909}%
\special{pa 1211 894}%
\special{pa 1176 880}%
\special{pa 1142 866}%
\special{pa 1108 854}%
\special{pa 1076 844}%
\special{pa 1044 836}%
\special{pa 1013 830}%
\special{pa 984 828}%
\special{pa 957 829}%
\special{pa 931 835}%
\special{pa 908 844}%
\special{pa 887 859}%
\special{pa 869 878}%
\special{pa 853 901}%
\special{pa 838 928}%
\special{pa 826 958}%
\special{pa 816 990}%
\special{pa 808 1025}%
\special{pa 801 1061}%
\special{pa 796 1099}%
\special{pa 793 1138}%
\special{pa 791 1177}%
\special{pa 791 1216}%
\special{pa 792 1254}%
\special{pa 794 1292}%
\special{pa 799 1329}%
\special{pa 805 1364}%
\special{pa 814 1397}%
\special{pa 825 1429}%
\special{pa 838 1457}%
\special{pa 854 1484}%
\special{pa 873 1507}%
\special{pa 894 1526}%
\special{pa 919 1542}%
\special{pa 946 1553}%
\special{pa 976 1562}%
\special{pa 1008 1567}%
\special{pa 1074 1575}%
\special{pa 1138 1583}%
\special{pa 1170 1588}%
\special{pa 1193 1592}%
\special{fp}%
%
\special{pn 8}%
\special{pa 1130 1689}%
\special{pa 2224 1689}%
\special{fp}%
\special{sh 1}%
\special{pa 2224 1689}%
\special{pa 2157 1669}%
\special{pa 2171 1689}%
\special{pa 2157 1709}%
\special{pa 2224 1689}%
\special{fp}%
%
\special{pn 8}%
\special{pa 1678 2070}%
\special{pa 1678 1270}%
\special{fp}%
\special{sh 1}%
\special{pa 1678 1270}%
\special{pa 1658 1337}%
\special{pa 1678 1323}%
\special{pa 1698 1337}%
\special{pa 1678 1270}%
\special{fp}%
\put(21.7800,-16.6500){\makebox(0,0)[lb]{$\xi$}}%
\put(17.0900,-13.0600){\makebox(0,0)[lb]{$\eta$}}%
\put(15.7000,-18.2000){\makebox(0,0)[lb]{P}}%
\put(8.1000,-11.9100){\makebox(0,0)[lb]{$\partial \Omega$}}%
%
\special{pn 8}%
\special{ar 1672 1689 317 319  0.0000000  6.2831853}%
%
\special{pn 20}%
\special{pa 1358 1696}%
\special{pa 1998 1696}%
\special{fp}%
\put(19.5000,-19.5000){\makebox(0,0)[lb]{$B_{\epsilon}(P)$}}%
\put(9.5000,-20.1000){\makebox(0,0)[lb]{Fig.1}}%
\end{picture}}%
\end{minipage}
Now 
\begin{align*}
\nu_{y}\log |x-y|
&=\frac{(\xi_2-\xi_1)F'(\xi_2)-(\eta_2-\eta_1)}
{\{(\xi_1-\xi_2)^2+(\eta_{1}-\eta_{2})^2\}\{1+(F'(\xi_2))^2\}^{1/2}} \\
&=\frac{(\xi_2-\xi_1)F'(\xi_2)-(F(\xi_2)-F(\xi_1))}
{\{(\xi_1-\xi_2)^2+(F(\xi_{1})-F(\xi_{2}))^2\}\{1+(F'(\xi_2))^2\}^{1/2}}. \\
\end{align*}
Since 
$$
F(\xi_1)-F(\xi_2)-(\xi_1-\xi_2)F'(\xi_2)=(\xi_1-\xi_2)^2\int_0^{1} t F''(\xi_2+(\xi_1-\xi_2)t)\; dt 
$$
and 
$$
F(\xi_1)-F(\xi_2)=(\xi_1-\xi_2)\int_0^1 F'(\xi_2+(\xi_1-\xi_2)t)\; dt, 
$$
we obtain 
\begin{align*}
\nu_{y}\log |x-y|
&=\frac{(\xi_1-\xi_2)^2\int_0^{1} t F''(\xi_2+(\xi_1-\xi_2)t)\; dt}
{[(\xi_1-\xi_2)^2+(\xi_1-\xi_2)^2\{\int_0^1 F'(\xi_2+(\xi_1-\xi_2)t)\; dt\}^2 ]
\{1+(F'(\xi_2))^2\}^{1/2}} \\
&=\frac{\int_0^{1} t F''(\xi_2+(\xi_1-\xi_2)t)\; dt}
{[1+\{\int_0^1 F'(\xi_2+(\xi_1-\xi_2)t)\; dt\}^2 ]
\{1+(F'(\xi_2))^2\}^{1/2}}. \\
\end{align*}
The positive denominator is of class $C^{k-1}$ and the numerator is of class $C^{k-2}$, 
including when $x=y$. Moreover 
$$
\lim_{\substack{x\rightarrow y \\[1pt] x \in \partial \Omega}} \nu_{y} \log\frac{1}{|x-y|}
=-\frac{\frac{1}{2} F''(\xi_2)}
{\{1+(F'(\xi_2))^2\}^{3/2}}=-\frac{1}{2}\kappa(y).  
$$ 
\end{proof}
For $p<2$, the Schatten class $S_p(L^2)$ cannot be characterized as in the case $p=2$ 
by a property analogous to the square integrability of integral kernels. To obtain criteria 
for operators to belong to Schatten classes for $p<2$, we use the result of J. Delgado and 
M. Ruzhansky:  
\begin{theorem}[\cite{DR} Theorem 3.6]
Let $M$ be a closed smooth manifold of dimension $n$ and let $\mu_1,\ \mu_2\geqq 0$. 
Let $K\in L^2(M\times M)$ be such that $E(x, y)\in H_{x, y}^{\mu_1, \mu_2}(M\times M)$. 
Then the integral operator $K$ on $L^2(M)$, defined by 
$$
(Kf)(x)=\int_{M} E(x,\ y) f(y)\; dy,
$$
is in the Schatten classes $S_{r}(L^2(M))$ for $r>\frac{2n}{n+2(\mu_1+\mu_2)}$.   
\end{theorem}
\begin{proof}[Proof of Theorem 2.12]
Taking a $C^{\infty}$ atlas on $M=\partial \Omega$ like Lemma 2.13, we see      
$$E(x, y)\in C_{x, y}^{k-2} (M\times M). $$ 
Let $n=\mbox{dim}\;\partial \Omega=1$ and $\mu_1+\mu_2=k-2$. From Theorem 2.14, we have 
$$
K\in S_r(L^2(M))\quad \mbox{for all}\ r>\frac{2}{2k-3}. 
$$
Using Weyl's inequality again (See e.g. \cite{Si}, \cite{Te}),  
$$\{\sum_{j=1}^{\infty} |\alpha_{j}|^r\}^{1/r} 
\geqq \{\sum_{\lambda_j\in \sigma_p(K)} |\lambda_j|^r\}^{1/r}.$$ 
The L.H.S. is the Shatten norm of $K$ which is finite.  
\end{proof}
\begin{corollary}
Let $n=2$ and $\Omega$ be a $C^{\infty}$ region.   
$$
\alpha_j=o(j^{-\infty})\ \mbox{and}\ \lambda_j = o(j^{-\infty}) \quad \text{as}\; j\rightarrow \infty.  
$$
\end{corollary}
\begin{remark} 
If $\Omega$ is a $C^6$ region, then $E(x, y)\in C_{x, y}^{2, 2}(M\times M)$. 
From \cite[Corollary 4.4]{DR}, $K$ is a trace class operator and its trace is given by  
$$
\sum_{\lambda_i \in \sigma_p(K)} \lambda_i \equiv {\rm tr}(K)
=\int_{\partial \Omega} -\frac{1}{2 \pi}\kappa(y) d{s_y}=-1. 
$$
$-1$ is an eigenvalue of $K$, so the sum of Fredholm eigenvalues is 0. 
\end{remark}
The $L^p \rightarrow L^q$ estimate of eigenfunctions is  
one of the main interests in spectral geometry.  
From Lemma 2.13, we find a fundamental estimate of eigenfunction:  
\begin{remark}
Let $n=2$ and $\Omega$ be a $C^2$ region. There exists a constant $C$, depending only on $\Omega$,  
such that  
$$ \Vert e_{\lambda_j}\Vert_{L^\infty(\partial \Omega)} 
\leqq C {\lambda_j}^{-1} \Vert e_{\lambda_j} \Vert_{L^1(\partial \Omega)}. $$
\end{remark}
Presumably this is the best $L^1\rightarrow L^{\infty}$ estimate of eigenfunctions. 
\section{Nodal sets of eigenfunctions}
Few studies have focused on the eigenfunctions. 
In this section, we introduce some fundamental 
estimates for nodal sets of eigenfunctions of 
two-dimensional double layer potentials.  
\subsection{Basic properties of nodal sets}
\par\hspace{5mm} 
The nodal set $N(e_{\lambda}(x))$ of eigenfunction $e_{\lambda}(x)$  
is defined by: 
$$
N(e_{\lambda}(x)) \equiv \{\; x\in \partial \Omega\; |\; e_{\lambda}(x)=0\; \}. 
$$
We note that the nodal set of non-constant eigenfunction is not empty: 
\begin{theorem} Let $\Omega$ be a bounded $C^{2}$ region in $\mathbf{R}^n$ 
and $0<\phi(x)\in C(\partial \Omega)$ be an eigenfunction of $K$. Then $\phi(x)={\rm const}$. 
\end{theorem}
This theorem holds true even for $n\geqq 3$. To prove Theorem 3.1,  
we closely follow \cite{KPS} and introduce 
the properties of symmetrizable operators. The proposition below is aimed 
at and will be directly applicable to double layer potentials $K$. 
We know that $K$ is in some Schatten classes 
(See \S 2 and \S 4 for the case of $n=3$).  
Moreover the eigenvalues of symmetrizable Schatten class 
operators are given by Min-Max methods (See e.g. \cite[\S3 and Proposition 3]{KPS}): 
%
\begin{proposition}[Min-Max principle for double layer potentials]
Let $\lambda_1^{+} \geqq \lambda_2^{+}\geqq \cdots 
\geqq 0 \geqq \cdots \geqq \lambda_1^{-} > \lambda_0^{-}=-1$
be the eigenvalues of $K$ repeated according to their mulitplicity, 
and let $\phi_k^{+},\ \phi_k^{-}$ be the correspoonding eigenfunctions. 
\par 
Then, 
$$
\lambda_k^{+}= {\rm max}_{f\perp \{\phi_1^{+}, \cdots, \phi_{k-1}^{+} \}} 
\frac{\langle SK f,\ f \rangle}{\langle Sf,\ f \rangle}, 
$$
and similarly 
$$
\lambda_k^{-}= {\rm min}_{f\perp \{\phi_0^{-}, \cdots, \phi_{k-1}^{-} \}} 
\frac{\langle SK f,\ f \rangle}{\langle Sf,\ f \rangle}.
$$ 
\end{proposition}
Here we may employ the single layer potential $S$  
defined by 
$$
(S \psi)(x) \equiv  \int_{\partial \Omega} E(x, y) \psi(y) \; dS_y   
$$
and $f \perp g$ means $\langle f, S g \rangle =0$. Especially if $\lambda\not=-1$, 
$e_{\lambda}(x) \in \{ \phi(x)\in L^2(\partial\Omega) 
\; |\; \langle \phi, S 1 \rangle=0\; \}$. 
\begin{proof}[Proof of Theorem 3.1]
From Min-Max principle for double layer potentials, 
non constant eigenfunctions $\{ e_{\lambda}(x)\}$ satisfy  
$e_{\lambda}(x) \in \{ \phi(x)\in L^2(\partial\Omega) 
\; |\; \langle \phi, S 1 \rangle=0\; \}$.  
Remarking that $f(x)=S1(x)>0$ for $n\geqq 3$ and 
$$
\int_{\partial\Omega} f(x)\phi(x) \; dS_x=0, 
$$  
there exists subset $N^{-}\subset \partial \Omega$ 
such that $\phi(x)<0$ on $N^{-}$. 
\par 
For $n=2$, eigenfunctions and eigenvalues are equivalent under the 
self-similar transformations.  
Indeed, letting $x_\epsilon=\epsilon x,\ y_{\epsilon}=\epsilon y$, 
$\Omega_{\epsilon}=\{ x_{\epsilon}\ |\ x \in \Omega\ \}$ and 
$\psi(x_\epsilon)\equiv \psi(x)$, we have 
$$
(K_\epsilon \psi)(x_{\epsilon}) \equiv  \int_{\partial \Omega_{\epsilon}} 
\psi(y_{\epsilon})\cdot \nu_{y_{\epsilon}} E(x_{\epsilon}, y_{\epsilon}) \; ds_{y_{\epsilon}} 
=
\int_{\partial \Omega} 
\psi(y)\cdot \nu_{y} E(x, y) \; ds_{y}=(K \psi)(x). 
$$ 
Since $S1(x)>0$ for the shrinking 
region $\Omega_\epsilon$. Again using the min-max principle, 
there exists subset $N^{-}\subset \partial \Omega$ 
such that $\phi(x)<0$ on $N^{-}$.   
\end{proof}
For convex region, we can give another short proof of Theorem 3.1 without Proposition 3.2. 
\begin{remark} Let $\Omega$ be a convex region in $\mathbf{R}^n$ 
and $\phi(x)>0$ be an eigenfunction of $K$. Then $\phi(x)={\rm const}$. 
\end{remark}
\begin{proof}
From a convex separation theorem, 
$$
\nu_y E(x, y)=C \frac{x-y}{|x-y|^{n-1}}\cdot {\mathbf n}_{y}  \leq 0 
\quad(\forall x,\ y \in \partial \Omega).
$$
Remarking that 
$
(K 1)(x) \equiv  \int_{\partial \Omega} \nu_{y} E(x, y) \; ds_y=-1  
$
and using the first mean value theorem for integration,   
for all $x \in \partial \Omega$ there exists $x' \in \partial \Omega$ satisfying 
$$
(K \phi)(x)=-\phi(x'). 
$$
For a non-constant eigenfunction $\phi(x)>0$, 
we know $(K\phi)(x)=\lambda \phi(x)$ with $|\lambda|<1$. 
Thus 
$$\inf_{x\in \partial \Omega}|(K\phi)(x)|
=\inf_{x\in \partial \Omega}|\lambda \phi(x)| 
< \inf_{x'\in \partial \Omega}|\phi(x')|=\inf_{x\in \partial \Omega}|(K\phi)(x)|. 
$$ 
This is a contradiction. 
\end{proof}
In the following example, the nodal set of second eigenfunction 
of double layer potential divides $\partial \Omega$ into many pieces.
\begin{example}
Let $\partial \Omega= S^1$. For an arbitary non-empty closed 
set $A \subsetneqq \partial \Omega$, 
there exists eigenfunction $e_0(x)\not=0$ such that 
$$
A \subset N(e_0(x)).  
$$  
\end{example}
\begin{proof}
From Corollary 2.8, we just choose the non-constant function 
$e_0(x)\in C(\partial \Omega)\cap L^2_0(\partial \Omega)$ to satisfy $e_0(x)=0$ on $A$.   
\end{proof}
We recall Courant's nodal line theorem (CNLT). 
CNLT states that 
if the eigenvalues $\lambda_n$ of Laplacian are ordered increasingly, 
then each eigenfunctions $u_n(x)$ corresponding to $\lambda_n$, divides 
the region by its nodal set, into at most $n$ subdomains. 
Unlike the CNLT, we find that the nodal set of double layer eigenfunction $e_n$ is 
characterized by not $n$ but $\lambda_n$. 
\subsection{Two dimensional analytic boundary}
In this subsection, we only consider the analytic domains 
$\Omega \subset {\mathbf R}^2$ and real analytic eigenfunctions $\{e_\lambda(x)\}\subset 
C^{\omega}(\partial \Omega)$. 
This assumption is reasonable since the continuous eigenfunction $e_{\lambda}(x)$ 
is also analytic for $\lambda\not=0$ (See Remark 3.8).  
\par
We prove the boundary zeroes $N(e_{\lambda}(x))$ satisfy   
$$ \sharp N(e_{\lambda}(x))< C |\log |\lambda||$$
where $K e_{\lambda}(x)=\lambda e_{\lambda}(x)$. 
\subsubsection{Holomorphic extentions of eigenfunctions} 
The following notations and results are heavily borrowed  
from Garabedian (See \cite{Ga}), Millar {(See \cite{Mi1}, \cite{Mi2}, \cite{Mi3})} and 
Toth-Zelditch (See \cite{TZ}): 
We denote points ${\mathbf{R}}^2$ and also in ${\mathbb{C}}^2$ by $(x, y)$. 
We further write $z=x+iy$, $z^*=x-iy$. Note that $z$, $z^*$ are independent 
holomorphic coordinates on ${\mathbb{C}}^2$ and are characteristic coordinates 
for the Laplacian $\frac{1}{4}\triangle$, in that Laplacian analytically 
extends to $\frac{\partial^2}{\partial z \partial z^*}$.  
When dealing with the kernel functions of two variables, we use $(\xi, \eta)$ 
in the same way as $(x, y)$ for the second variable. 
\par 
When the boundary is real analytic, the complexification $\partial \Omega \subset{\mathbb C}$ 
is the image of analytic continuation of a real analytic parametrization. 
For simplicity and without loss of generality, we will assume that 
the length of $\partial \Omega=2\pi$. We denote a real parametrization 
by arc-length by 
$Q :  {\mathbf{S}}^1 \rightarrow \partial \Omega \subset{\mathbb C}$, 
and also write the parametrization as a periodic function 
$$
q(t)=Q(e^{it})\; : \; [0, 2\pi]\rightarrow \partial \Omega
$$
on $[0, 2\pi]$. We then put the complex conjugate by $q(s)=q_1(s)+iq_2(s)$, $\bar{q}(s)=q_1(s)-iq_2(s)$ 
for $s\in [0, 2\pi]$. 
\par 
We complexify $\partial \Omega$ by holomorphically extending the parametrization 
to $Q^{\mathbb{C}}$ on the annulus 
$$
A(\epsilon) \equiv \{ \tau \in {\mathbb{C}}\; :\; e^{-\epsilon}<|\tau|<e^{\epsilon}\; \}
$$
for $\epsilon>0$ small enough. 
Note that the complex conjugate parametrization $\bar{Q}$ extends holomorphically 
to $A(\epsilon)$ as $Q^{*\hspace{0.3mm} {\mathbb {C}}}$. The $q(t)$ parametrization 
analytically continues to a periodic function $q^{\mathbb{C}}(t)$ on 
$[0, 2\pi]+i[-\epsilon, \epsilon]$. The complexification $\partial \Omega_{\mathbb C}(\epsilon)$ 
of $\partial\Omega$ is denoted by 
$$
\partial \Omega_{\mathbb C}(\epsilon) \equiv Q^{\mathbb{C}}(A(\epsilon))\; \subset{\mathbb C}. 
$$
\par 
Next, we put $r^2((x, y); (\xi, \eta))=(\xi-x)^2+(\eta-y)^2$. 
For $s\in {\mathbf{R}}$ and $t\in {\mathbb{C}}$, we have 
$q(s)=\xi(s)+i \eta(s)$, $q^{\mathbb{C}}(t)=x(t)+iy(t)$, $q^{\mathbb{C}\hspace{0.3mm} *}(t)=x(t)-iy(t)$ and 
we write $r^2(q(s); q^{\mathbb{C}}(t))$. Thus 
$$
r^2((x, y); (\xi, \eta))=(q(s)-q^{\mathbb{C}}(t))(\bar{q}(s)-q^{{\mathbb{C}} \hspace{0.3mm} *}(t)) 
\; \in {\mathbb{C}}.
$$
To clarify the notation, we consider two examples:
\begin{example}[The circle]
Let $\partial \Omega=S^1$. Then, $q(s)=e^{is}$, $t=\theta+i\xi$, 
$q^{\mathbb{C}}(t)=e^{i(\theta+i\xi)}$, $q^{\mathbb{C}\hspace{0.3mm} *}(t)=e^{-i(\theta+i\xi)}$, 
$\bar{q}^{\mathbb{C}\hspace{0.3mm} *}(t)=e^{i(\theta-i\xi)}$, and 
$$
r^2(s, t)=(e^{i(\theta+i\xi)}-e^{is})(e^{-i(\theta+i \xi)}-e^{-is})= 4\sin^{2}\frac{\theta-s+ i\xi}{2}.
$$
Thus, $\log{r^2}=\log(4\sin^{2}\frac{\theta-s+ i\xi}{2}).$  
\end{example}
\begin{example}[The ellipse]
Let $\partial \Omega=\{(x, y)\; |\; \frac{x^2}{9}+y^2=1\ \}$. Then, $q(s)=2e^{is}+e^{-is}$, 
$t=\theta+i\xi$, $q^{\mathbb{C}}(t)=2e^{i(\theta+i\xi)}+e^{-i(\theta+i\xi)}$, 
$q^{\mathbb{C}\hspace{0.3mm} *}(t)=2e^{-i(\theta+i\xi)}+e^{i(\theta+i\xi)}$, 
$\bar{q}^{\mathbb{C}\hspace{0.3mm} *}(t)=2e^{i(\theta-i\xi)}+e^{-i(\theta-i\xi)}$, and 
$$
r^2(s, t)=(2e^{i(\theta+i\xi)}+e^{-i(\theta+i\xi)}-e^{is})
(2e^{-i(\theta+i\xi)}+e^{i(\theta+i\xi)}-e^{-is}). 
$$
\end{example}
We denote by $\frac{\partial}{\partial n}$ the not-necessarity-unit normal 
derivative in the direction $iq'(s)$. Thus, in terms of the notation 
$\frac{\partial}{\partial \nu}$ above, $\frac{\partial}{\partial n}
=|q'(s)|\frac{\partial}{\partial \nu}$. When we are using an arc-length parametrization, 
$\frac{\partial}{\partial n}=\frac{\partial}{\partial \nu}$. 
One has 
$$
\frac{d}{ds}\log r=\frac{1}{2}\Big[\frac{q'(s)}{q(s)-q^{\mathbb{C}}(t)}
+\frac{\bar{q}'(s)}{\bar{q}(s)-q^{\mathbb{C} \hspace{0.3mm} *}(t) } \Big],\ 
\frac{\partial}{\partial n} \log r=\frac{-i}{2}\Big[
\frac{q'(s)}{q(s)-q^{\mathbb{C}}(t)}
-\frac{\bar{q}'(s)}{\bar{q}(s)-q^{\mathbb{C} \hspace{0.3mm} *}(t) } \Big]. 
$$
\par 
\subsubsection{Analytic continuation of eigenfunctions through layer potential representation}
Since $r^2(s, t)=0$ when $s=t$, the logarithtic factor in $K$ now gives rise to 
a multi-valued integrand. Neverthless any derivative of $\log r^2$ is unambiguously defined and  
the analytic continuation of complex representation 
was given by Millar (See \cite[p.508 (7.2)]{Mi1}): 
\begin{proposition}
The integral $K e_{\lambda}(q(s))=\frac{1}{\pi} \int_{0}^{2\pi} e_{\lambda}(q(s)) 
\frac{\partial}{\partial \nu}\log r(s, t)\; ds
=\frac{1}{\pi}\int_{0}^{2\pi} e_{\lambda}(q(s)) 
\frac{1}{r}\frac{\partial r}{\partial \nu}(s, t)\; ds$ is real analytic on the parameter interval 
$S^1$ parametrizing $\partial \Omega$ and holomoriphically extended to an annulus $A(\epsilon)$ 
by the formula 
$$
Ke_{\lambda}(q^{\mathbb{C}}(t))=
\frac{1}{2\pi i} \int_{0}^{2\pi} e_{\lambda}(q(s)) \left(\frac{q'(s)}{q(s)-q^{\mathbb{C}}(t)}
-\frac{\bar{q}'(s)}{\bar{q}(s)-q^{\mathbb{C} \hspace{0.3mm} *}(t)}\right)\; ds.
$$
\end{proposition}
\begin{proof}
We first remark that 
$\frac{\partial}{\partial \nu}=|q'(s)|^{-1}\frac{\partial}{\partial n}$, 
so the integral representation is invariant under reparametrization. 
Any derivative of $\log r^2$ is unambiguously defined and we already have 
$$
\frac{1}{r} \frac{\partial r}{\partial n}=\frac{\partial \log r}{\partial n}
=\frac{1}{2i}\Big[\frac{q'(s)}{q(s)-q^{\mathbb{C}}(t)}
-\frac{\bar{q}'(s)}{\bar{q}(s)-q^{\mathbb{C} \hspace{0.3mm} *}(t) } \Big]. 
$$
In the real domain $q^{\mathbb{C} \hspace{0.3mm}*}(t)=\bar{q}^{\mathbb{C} }(t)$, so 
$$\frac{1}{r}\frac{\partial r}{\partial n}=\mbox{Im} \frac{q'(s)}{q(s)-q(t)}. $$
Here  $\mbox{Im}\; z$ denotes the imaginary part of $z$. 
We recall that in terms of the real parametrization, 
$\frac{1}{r} \frac{\partial r}{\partial \nu}$ is real and continuous (See Lemma 2.13). 
\par
In complex notation, the same statement follows from the fact that 
$$
\lim_{t\rightarrow s}\frac{q(s)-q^{\mathbb{C}}(t)}{s-t}=q'(s) \quad 
\Rightarrow \quad 
\frac{q'(s)}{q(s)-q^{\mathbb{C}}(t)}=\frac{1}{s-t}+O(1),\ (s\rightarrow t), 
$$
where $\frac{1}{s-t}$ is real when $s, t \in {\mathbf{R}}$. 
Hence 
$\mbox{Im} \frac{q'(s)}{q(s)-q^{\mathbb{C}}(t)}$ is continuous for $s, t\in [0, 2\pi]$ and 
since $q(s),\; q(t)$ are real analytic, the map 
$$ 
s\; \rightarrow\; \Big[\frac{q'(s)}{q(s)-q^{\mathbb{C}}(t)}
-\frac{\bar{q}'(s)}{\bar{q}(s)-q^{\mathbb{C} \hspace{0.3mm} *}(t) } \Big]
$$ 
is a continuous map from  $s\in [0, 2\pi]$ to the space of holomorphic functions of $t$. 
So the integral admits an holomorphic extention. 
\end{proof}
\begin{remark}
We notice that the continuous eigenfunction satisfies 
$$
e_{\lambda}(q(s))=\frac{1}{\lambda} Ke_{\lambda}(q(s)) \quad \mbox{for}\; s\in  {\mathbf{S}}^1.  
$$
From Proposition 3.7, if $\lambda\not=0$, the continuous eigenfunction is also analytic.
\end{remark}
\subsubsection{Growth of zeroes and Growth of $e_{\lambda}^{\mathbb{C}}(q^{\mathbb{C}}(t))$}
The main purpose of this subsection is to give an upper bound for the number of complex 
zeroes of $e_{\lambda}^{\mathbb{C}}$ in $\partial \Omega_{\mathbb{C}}(\epsilon)$ 
in terms of the growth of $|e_{\lambda}^{\mathbb{C}}(q^{\mathbb{C}}(t))|$. 
For the eigenvalue $\lambda$ and for a region $D\subset \partial \Omega_{\mathbb{C}}(\epsilon)$ 
we denote by 
$$
n(\lambda, D)=\sharp \{ q^{\mathbb{C}}(t)\in D\; :\; e_{\lambda}^{\mathbb{C}}(q^{\mathbb{C}}(t))=0 \}. 
$$
To the reader's convenience, we recall that the classical distribution theory of holomorphic 
functions is concerned with the relation between the growth of the number of zeroes of 
a holomorphic function $f$ and the growth of $\mbox{max}_{|z|=r} \log|f(z)|$ on discs of 
increasing radius. The following estimate, suggested by Lemma 6.1 of Donnelly-Fefferman 
(See \cite{DF}), gives an upper bound on the number of zeroes in terms of the growth of the family: 
\begin{proposition}
Normalize $e_{\lambda}$ so that $\Vert e_{\lambda}\Vert_{L^2(\partial \Omega)}=2\pi$. Then 
there exists a constant $C(\epsilon)>0$ such that for any $\epsilon>0$, 
\def\Max{\ensuremath\mathop\mathrm{max}}
$$
n(\lambda, \partial \Omega_{\mathbb{C}}(\epsilon/2)) \leqq C(\epsilon) 
\max_{q^{\mathbb{C}}(t)\in  \partial \Omega_{\mathbb{C}}(\epsilon)} 
\Big|\log | e_{\lambda}^{\mathbb{C}}(q^{\mathbb{C}}(t)) |\Big|.
$$
\end{proposition}
\begin{proof} 
Let $G_{\epsilon}$ denote the Dirichlet Green's function of 
$\frac{2}{\pi}\frac{\partial^2}{\partial z \partial z^*}$ in the `annulus' 
$\partial \Omega_{\mathbb{C}}(\epsilon)$. 
Also, let $\{a_k\}_{k=1}^{n(\lambda, \partial \Omega_{\mathbb{C}}(\epsilon/2))}$ 
denote the zeroes of $e_{\lambda}^{\mathbb{C}}$ 
in the sub-annulus $\partial \Omega_{\mathbb{C}}(\epsilon/2)$. 
Let $f_{\lambda}=\frac{e_{\lambda}^{\mathbb{C}}}{ \Vert e_{\lambda}^{\mathbb{C}}\Vert_{\partial \Omega_{\mathbb{C}}(\epsilon)} } $
where $\Vert u \Vert_{\partial \Omega_{\mathbb{C}}(\epsilon)} 
=\mbox{max}_{\zeta\in \partial \Omega_{\mathbb{C}}(\epsilon)} |u(\zeta)|. $ 
Then $\log|f_{\lambda}(q^{\mathbb{C}}(t))|$ can be separated into two terms: 
\begin{align*}
\log|f_{\lambda}(q^{\mathbb{C}}(t))| &=\int_{\partial \Omega_{\mathbb{C}}(\epsilon/2)} G_{\epsilon}(q^{\mathbb{C}}(t) , w)
\frac{i}{\pi} \partial \bar{\partial} \log |e_{\lambda}^{\mathbb{C}}(w)| +F_{\lambda}(q^{\mathbb{C}}(t)) \\
&=\sum_{a_k\in \partial \Omega_{\mathbb{C}}(\epsilon/2)\; :\; e_{\lambda}^{\mathbb{C}}(a_k)=0} G_{\epsilon}(q^{\mathbb{C}}(t), a_k)
+F_{\lambda}(q^{\mathbb{C}}(t)), 
\end{align*}
since 
$\frac{i}{\pi} \partial \bar{\partial} \log |e_{\lambda}^{\mathbb{C}}(w)| 
= \sum_{a_k\in \partial \Omega_{\mathbb{C}}(\epsilon/2)\; :\; e_{\lambda}^{\mathbb{C}}(a_k)=0} 
\delta_{a_k}$ 
which is called Poincar\'e-Lelong formula of holomorphic functions (See e.g. \cite[p.9 (3.6)]{De}). 
Moreover the function $F_{\lambda}$ is subharmonic on $\partial \Omega_{\mathbb{C}}(\epsilon)$ 
in the sense of distribution:   
$$
\frac{i}{\pi}\partial \bar{\partial} F_{\lambda} 
= \frac{i}{\pi}\partial \bar{\partial} \log|f_{\lambda}(q^{\mathbb{C}}(t))|
-\sum_{a_k\in \partial \Omega_{\mathbb{C}}(\epsilon/2)\; :\; e_{\lambda}^{\mathbb{C}}(a_k)=0} 
\frac{i}{\pi} \partial \bar{\partial} G_{\epsilon}(q^{\mathbb{C}}(t), a_k)
=\sum_{a_k\in \partial \Omega_{\mathbb{C}}(\epsilon) \backslash \partial \Omega_{\mathbb{C}}(\epsilon/2) \; :\; e_{\lambda}^{\mathbb{C}}(a_k)=0} \delta_{a_k}>0.  
$$ 
So, by the maximum principle for subharmonic functions, we obtain 
$$
\max_{\partial \Omega_{\mathbb{C}}(\epsilon)} F_{\lambda}(q^{\mathbb{C}}(t)) 
\leqq
\max_{\partial(\partial \Omega_{\mathbb{C}}(\epsilon))} F_{\lambda}(q^{\mathbb{C}}(t)) 
=\max_{\partial(\partial \Omega_{\mathbb{C}}(\epsilon))} \log |f_{\lambda}(q^{\mathbb{C}}(t))|=0.
$$
It follows that 
$$
\log |f_{\lambda}(q^{\mathbb{C}}(t))| \leqq 
\sum_{a_k\in \partial \Omega_{\mathbb{C}}(\epsilon/2) \; :\; e_{\lambda}^{\mathbb{C}}(a_k)=0} 
G_{\epsilon}(q^{\mathbb{C}}(t), a_k), 
$$
hence that 
$$
\max_{q^{\mathbb{C}}(t)\in \partial \Omega_{\mathbb{C}}(\epsilon)}
\log |f_{\lambda}(q^{\mathbb{C}}(t))| \leqq 
\left( \max_{z, w \in \partial \Omega_{\mathbb{C}}(\epsilon/2)} G_{\epsilon}(z, w) \right) n(\lambda, \partial \Omega_{\mathbb{C}}(\epsilon/2)).  
$$
Now $G_{\epsilon}(z, w)\leqq \mbox{max}_{w \in \partial (\partial \Omega_{\mathbb{C}}(\epsilon))}G_{\epsilon}(z, w)=0$ 
and so $G_{\epsilon}(z, w)<0$ for $z, w \in \partial \Omega_{\mathbb{C}}(\epsilon/2)$. 
It follows that 
there exists a constant $\nu(\epsilon)<0$ 
so that $\mbox{max}_{z, w\in \partial \Omega_{\mathbb{C}}(\epsilon/2)}G_{\epsilon}(z, w) 
\leqq \nu(\epsilon)$. 
Hence, 
$$
\max_{q^{\mathbb{C}}(t)\in \partial \Omega_{\mathbb{C}}(\epsilon/2)}
\log |f_{\lambda}(q^{\mathbb{C}}(t))| \leqq 
\nu(\epsilon) n(\lambda, \partial \Omega_{\mathbb{C}}(\epsilon/2)).  
$$
Since both sides are negative, we obtain 
\begin{align*}
n(\lambda, \partial \Omega_{\mathbb{C}}(\epsilon/2)) 
&\leqq \frac{1}{|\nu(\epsilon)|} \Big|\max_{q^{\mathbb{C}}(t)\in \partial \Omega_{\mathbb{C}}(\epsilon/2)} 
  \log |f_{\lambda}(q^{\mathbb{C}}(t))| \Big| \\
&\leqq \frac{1}{|\nu(\epsilon)|} \Big( \max_{q^{\mathbb{C}}(t)\in \partial \Omega_{\mathbb{C}}(\epsilon)} 
\log |e_{\lambda}^{\mathbb{C}}(q^{\mathbb{C}}(t)) |
-\max_{q^{\mathbb{C}}(t)\in \partial \Omega_{\mathbb{C}}(\epsilon/2)} 
\log |e_{\lambda}^{\mathbb{C}}(q^{\mathbb{C}}(t)) |  \Big)  \\ 
&\leqq \frac{1}{|\nu(\epsilon)|} \max_{q^{\mathbb{C}}(t)\in \partial \Omega_{\mathbb{C}}(\epsilon)} 
\log |e_{\lambda}^{\mathbb{C}}(q^{\mathbb{C}}(t)) |, \\
\end{align*}
where in the last inequality we use that $\mbox{max}_{q^{\mathbb{C}}(t)\in \partial \Omega_{\mathbb{C}}(\epsilon/2)} 
\log |e_{\lambda}^{\mathbb{C}}(q^{\mathbb{C}}(t))|\geqq 0$, which holds since 
$|e_{\lambda}^{\mathbb{C}}|\geqq 1$ at some point in $\partial \Omega_{\mathbb{C}}(\epsilon/2)$. 
Indeed, by our normalization, $\Vert e_{\lambda}\Vert_{L^2(\partial \Omega)}=2\pi$, and 
so there must already exist a points on $\partial \Omega$ with $|e_{\lambda}|>1. $ 
Putting $C(\epsilon)=\frac{1}{|\nu(\epsilon)|}$ we have the desired result. 
\end{proof}
We obtain the main theorem:   
\begin{theorem}
Let $\Omega \subset {\mathbf R}^2$ be a real analytic domain and $|\lambda|\not=0$. 
For real analytic eigenfunctions $e_\lambda(x)$ we have  
$$ \sharp N(e_{\lambda}(x))< C |\log |\lambda||.$$ 
\end{theorem}
\begin{proof}
For real $t\in S^1=[0, 2\pi]/\sim$,     
$$
e_{\lambda}(q(t))=\frac{1}{\lambda} K e_{\lambda}
=\frac{1}{\lambda}\int_{0}^{2\pi} e_{\lambda}(q(s)) 
\frac{1}{r}\frac{\partial r}{\partial \nu}(s, t)\; ds. 
$$
The holomorphic extention of $e_\lambda(q(s)) \in C^{\omega}(S^1)$ to $C^{\omega}(A(\epsilon))$ is 
unique and hence from Proposition 3.7,  
$$ e_{\lambda}^{\mathbb{C}}(q^{\mathbb{C}}(t)) 
=\frac{1}{2\pi i \lambda} 
\int_{0}^{2\pi} e_{\lambda}(q(s)) \left(\frac{q'(s)}{q(s)-q^{\mathbb{C}}(t)
}-\frac{\bar{q}'(s)}{\bar{q}(s)-q^{\mathbb{C} \hspace{0.3mm} *}(t)}\right)\; ds. 
$$
Remarking that the function $\left( \cdots \right)$ is continuous 
and bounded on $A(\epsilon)$ from the proof of Proposition 3.7. 
So using Cauchy-Schwarz inequality, there exists $C_{A(\epsilon)}>0$ such that  
\begin{align*}
|e_{\lambda}^{\mathbb{C}}(q^{\mathbb{C}}(t))| 
& \leqq \Big| \frac{1}{2\pi i \lambda} \Big| \cdot 
\Vert e_{\lambda}(q(s))\Vert_{L^2 (\partial \Omega)} 
\int_{0}^{2\pi} \Big| \frac{q'(s)}{q(s)-q^{\mathbb{C}}(t)
}-\frac{\bar{q}'(s)}{\bar{q}(s)-q^{\mathbb{C} \hspace{0.3mm} *}(t)} \Big|^2 \; ds \\ 
& \leqq \Big|\frac{1}{\lambda} \Big| \cdot C_{A(\epsilon)} \cdot  
\Vert e_{\lambda}(q(s))\Vert_{L^2(\partial \Omega)}.  
\end{align*}
Letting $\Vert e_{\lambda}(q(s))\Vert_{L^2(\partial \Omega)}=2\pi$ and by Proposition 3.9, 
\begin{align*}
n(\lambda, \partial \Omega_{\mathbb{C}}(\epsilon/2)) & \leqq C(\epsilon) 
\max_{q^{\mathbb{C}}(t)\in \partial \Omega_{\mathbb{C}}(\epsilon)} 
\Big|\log | e_{\lambda}^{\mathbb{C}}(q^{\mathbb{C}}(t)) |\Big| \\
& \leqq C(\epsilon)\Big| \log \Big[ \Big|\frac{1}{\lambda} \Big| \cdot C_{A(\epsilon)} \cdot  
\Vert e_{\lambda}(q(s))\Vert_{L^2(\partial \Omega)}\Big]\Big| \\ 
& \leqq \tilde{C}(\epsilon) |\log |\lambda|| 
\end{align*}
as desired. 
\end{proof}
\section{Double layer potentials in $\mathbf{R}^3$}
\par\hspace{5mm} 
Plemelj \cite{Pl} derived a fundamental result on the double layer potential 
in $\mathbf{R}^3$ which states that the eigenvalues of $K$ satisfy the following inequality 
$$
-1\leq \lambda_j<1
$$
For the case of a sphere, however, it is known that the eigenvalue of $K$ 
are negative, and by a straightforward calculation it can be shown 
that the eigenvalues are given by 
$$
\lambda_j=-\frac{1}{2j+1}, \quad (j=0, 1, 2, \cdots)
$$
with multiplicity $2j+1$. So $\sigma_p(K)$ for $n=3$ 
is very different from it for $n=2$ (See example 2.1, 
example 2.2 and Theorem 2.7).
Furthermore, Ahner and Arenstrof \cite{AA} have shown that when $\partial \Omega$ is 
a prolate spheroid, the corresponding eigenvalues are also negative. 
Consequently, for this geometry, the spectrum of $K$ also lies in the 
closed interval $[-1, 0]$.  
\par
Apart from these calculations, for the case of a special oblate spheroid,  
Ahner \cite[p.333]{Ah2} finds the positive eigenvalue 
$\lambda= 0.0598615\cdots <1$.  
This is an example of positive eigenvalues. Unfortunately, 
this supremum of eigenvalues becomes a formidable task for general 
region. 
\par 
Neverthless, for $\bar{\lambda}= \sup\{\lambda_j\ |\ \lambda_j \in \sigma_p(K) \}, $
we know the supremum of the boundary variation 
(\cite[Lemma 3.2, Theorem 3.4]{ADR})
$$
\sup\limits_{\partial \Omega} \overline{\lambda}=1
$$
where the supremum is taken over all $C^{\infty}$ domain $\Omega$. 
Letting $\underline{\lambda}= \inf\{\lambda_j\ |\ \lambda_j \in \sigma_p(K)\backslash \{-1\} \}, $
we also know ( \cite[Lemma 3.2]{ADR}, \cite[Theorem 5]{KPS})
$$
\inf\limits_{\partial \Omega} \underline{\lambda}=-1.
$$
Here we introduce a result about $\underline{\lambda}$ : 
Steinbach and Wendland prove that 
$$ 
(1-\sqrt{1-c_0})\Vert w \Vert_{S^{-1}} 
\leqq \Vert (I\pm K) w \Vert_{S^{-1}} 
\leqq (1+\sqrt{1-c_0})\Vert w \Vert_{S^{-1}}
$$ 
where $c_0=\inf\limits_{w \in H^{1/2}_*}\frac{\langle Kw, w\rangle}{\langle S^{-1} w, w\rangle}$ and 
$\Vert w \Vert_{S^{-1}}=\sqrt{\langle S^{-1} w , w \rangle}_{L^2{(\partial \Omega)}}$ 
for $w\in H^{1/2}(\partial \Omega)$. They show $c_0\leqq 1$. 
(These constants are slightly different from those in the original papers. 
For more information see \cite[Theorem 3.2]{SW}.)  
Especially for the negative eigenvalue $\underline{\lambda}$
$$
\underline{\lambda} \geqq -\sqrt{1-c_0}. 
$$
Thus the shape dependent constant $c_0$ controls the eigenvalue $\underline{\lambda}$. 
Note that Pechstein recently gives the lower bound of $c_0$ by 
using the isoperimetric constant $\gamma(\Omega)$ 
and Sobolev extention constants (See \cite[Corollary 6.14]{Pe}, \cite{KRW}). 
In the case of $\partial\Omega=S^2$, 
$c_0=\frac{8}{9}$ and $\underline{\lambda}=-\sqrt{1-c_0}=-\frac{1}{3}$. 
\subsection{Asymptotic properties of $\sigma(K)$ for $n=3$}
\par 
For the case of $n=3$, 
D. Khavinson, M. Putinar and H. S. Shapiro briefly mentioned only a result: 
$K$ is in the Schatten class ${\mathcal{S}}_p(L^2(\partial \Omega))$, 
$p>2$ (See \cite[p.150]{KPS}). We shall explain it in more detail for smooth $\partial\Omega$. 
Following \cite[p.303]{Ke}, the nature of the diagonal singularity of the kernel 
$\nu_{y} E(x, y)$ shows that  
$$
E_{2}(x, y)=\int_{\Omega_z} \nu_{z} E(z, x)\cdot \nu_{z} E(z, y) dS_z=A(x, y)+B(x, y) \log(|x-y|) 
$$
where $A(x, y), B(x, y)\in C^{\infty}(\partial \Omega\times \partial \Omega)$. 
Since $E_2(x, y) \in H_{x, y}^{\mu_1, \mu_2} $ with $\mu_1+\mu_2<1$, 
applying Theorem 2.14 to $E_2$ 
$$
K^{*} K \ \in {\mathcal{S}}_r(L^2(\partial \Omega))\quad \mbox{for}\ r>\frac{4}{2+2}=1. 
$$ 
This means that $K$ is in the Schatten class 
${\mathcal{S}}_p(L^2(\partial \Omega))$, $p>2$. 
\par 
We note that the regularity of $A(x, y)$ and $B(x, y)$ is essential to the above result.   
Immediately a decay rate of $\sigma_p(K)$ is obtained:
\begin{remark}
Let $n=3$ and $\Omega$ be a smooth region. For $\alpha>-\frac{1}{2}$, 
$$
\lambda_{j}=o({j^{\alpha}}) \quad \text{as}\; j\rightarrow \infty. 
$$
\end{remark}
In the case of a sphere, this is the best possibile estimate.  
\subsection{Isoperimetric properties of $K$}
We want to characterize the isoperimetric properties by $\sigma_p{(K)}$. 
For the case of $n=3$, however, the explicit formula have not been obtained yet. 
In this subsection, some expected properties and conjectures are introduced.   
Seeing the case of $n=2$, $-\underline{\lambda}$ 
and Schatten norm are expected to minimize by $\partial\Omega=S^2.$ So we expect the 
following conjectures:
\setcounter{conjecture}{0} 
\begin{conjecture} Let $n=3$ and $\underline{\lambda} \equiv 
{\rm min}\; \sigma_p(K)\backslash \{-1\}$.  
We have  
$$
\sup\limits_{\partial \Omega} \underline{\lambda}=-\frac{1}{3},  
$$
where the supremum is taken over all $C^{\infty}$ simply 
connected closed surfaces. The supremum is achieved if and only if $\partial \Omega=S^2$. 
\end{conjecture}
Note that for the case of $\partial \Omega=S^2$,   
$\underline{\lambda}=-\frac{1}{3}$ is obtained by direct calculations. 
\begin{conjecture}
Let $n=3$. For $p>1$, we have  
$$
\inf\limits_{\partial \Omega} {\rm tr}\{(K^*K)^p\}=\left(1-\frac{1}{2^{2p-1}}\right)\zeta (2p-1)
$$
where the infimum is taken over all $C^{\infty}$ simply 
connected closed surfaces and $\zeta(x)$ denotes the Riemann zeta function. 
The infimum is achieved if and only if $\partial \Omega=S^2$. 
\end{conjecture}
To confirm the validity of conjectures, henceforce, we consider the case of ellipsoids. 
For the case of ellipsoids 
$\{ (x, y, z)\in {\mathbf{R}^3}\ |\ \frac{x^2}{a^2}+\frac{y^2}{b^2}+\frac{z^2}{c^2}=1 \}$ 
Ritter \cite{Ri1}\cite{Ri2} has shown that $\sigma_p(K)$ is completely solved by ellipsoidal harmonics 
(Lam\'e polynomials); note that there are exactly $2l+1$ linearly independent Lam\'e polynomials 
of order $l\geq 0$ (See \cite{H}). Also Martensen \cite[Theorem 1]{Ma} proved : 
\begin{proposition}
For any $2l+1$ linearly independent Lam\'e polynomials of order $l\geq0$, considered as eigenfunctions 
of $K$, the sum of corresponding eigenvalues is equal $-1$. 
\end{proposition}
We denote these eigenvalues by $\lambda_{k, l}$ $(k=1, 2, \cdots, 2l+1)$ and so 
$$
\sum_{k=1}^{2l+1} \lambda_{k, l}=-1. 
$$
Furthermore, deformation of the sphere into a triaxial ellipsoid yields to bifurcation 
$-\frac{1}{2l+1}$ into $2l+1$ different eigenvalues of order $l$, say $\lambda_{k,l}$, 
$k=1,2,\cdots, 2l+1,$ each with multiplicity one (See \cite{Ri2}).
\par
Consequently proofs of conjectures for ellipsoids are given: 
\begin{theorem}
Let $n=3$ and $p>1$. For the case of ellipsoids $\partial \Omega$, we have 
$$
\sup\limits_{\partial \Omega} \underline{\lambda}=-\frac{1}{3}  
\quad \mbox{and} \quad 
\inf\limits_{\partial \Omega} {\rm tr}\{(K^*K)^p\}=\left(1-\frac{1}{2^{2p-1}}\right)\zeta (2p-1).
$$
The supremum and infimum are achieved if and only if $\partial \Omega=S^2$.  
\end{theorem}
\begin{proof}
For $l=1$, 
$$\lambda_{1,1}+\lambda_{2,1}+\lambda_{3,1}=-1.$$  
Thus  
$\underline{\lambda} \leqq \mbox{min}(\lambda_{1,1}, \lambda_{2,1}, \lambda_{3,1}) \leqq -\frac{1}{3}.$
Equality holds if and only if $\lambda_{1,1}=\lambda_{2,1}=\lambda_{3,1}=-\frac{1}{3}$. So 
we have the first equation. 
\par
To prove the second equation, we note that from H\"order's inequality   
$$
1=|(1, 1, \cdots, 1)\cdot (\lambda_{1, l}, \lambda_{2, l}, \cdots, \lambda_{2l+1, l})|
\leqq 
(2l+1)^{(2p-1)/2p} (\sum_{k=0}^{2l+1} |\lambda_{k, l}|^{2p})^{1/2p}.  
$$
This leads to $s_l=\sum\limits_{k=0}^{2l+1} |\lambda_{k, l}|^{2p} 
\geqq \left(\frac{1}{2l+1}\right)^{2p-1}.$
Remarking that $(K^*K)^p$ is in trace class and using Weyl's inequality,  
\begin{align*}
{\rm tr}\{(K^*K)^p\} &\geqq \sum_{l=0}^{\infty} \sum_{k=0}^{2l+1} |\lambda_{k, l}|^{2p} \\
                     &= \sum_{l=0}^{\infty} s_l \\
                     &\geqq \sum_{l=0}^{\infty}  \left(\frac{1}{2l+1}\right)^{2p-1}
                      =\left(1-\frac{1}{2^{2p-1}}\right)\zeta(2p-1)    
\end{align*}
as desired. 
\end{proof}
For the general smooth surfaces, we mention equvalent statements of conjectures. 
We infer the Schatten norm of single layer potentials:  
$${\rm tr}\{(K^*K)^p\} \leqq {\rm tr} \{(\frac{1}{4}S^*S)^p\}
   + {\rm tr} [\{(K-\frac{1}{2}S)^*(K-\frac{1}{2}S)\}^p].$$
It's also known that 
(\cite[Theorem 8]{KPS} and see also \cite{EKS}, \cite{Re}, \cite{Ra1} and \cite{Ra2}) : 
\begin{theorem}
The following is true: for a ball in $\mathbf{R}^3$ the kernel of $K$ 
is symmetric and $K=\frac{1}{2}S$, and balls are the only domains with this property.
\end{theorem} 
Thus if one proves the single layer version of the above conjectures, 
simultaneously we obtain the proof for the double layer potentials.   
\section{Conclusion}
\par\hspace{5mm} 
Some fundamental properties of the eigenvalue and eigenfunctions of 
double layer potentials are discussed. 
Characteristic properties of the ball are given by the Hilbert-Schimidt norm and Schatten norms 
of double layer potentials. The fundamental estimates of decay rates of eigenvalues are also 
given by the regularity of the boundary. 
\par
With respect to eigenfunctions, the growth rates of nodal sets are 
characterized of the eigenvalues. Even less is known in $n=3$. 
We want to mention about this in the future. 

\end{document}